\documentclass[11pt]{amsart}
\usepackage{graphicx}
\usepackage{latexsym}
\usepackage{amsfonts,amsmath,amssymb}
\newtheorem{theorem}{Theorem}[section]																
\newtheorem{lemma}[theorem]{Lemma}

\newtheorem{corollary}[theorem]{Corollary}

\usepackage{color}

\def\eps{\varepsilon}

\def\a{\alpha}
\def\be{\beta}

\def\ga{\gamma}
\def\part{\partial}

\def\b1{\bold 1}

\newcommand{\beq}{\begin{equation}}
\newcommand{\eeq}{\end{equation}}

\theoremstyle{remark}

\numberwithin{equation}{section}
\date{\today}

\begin{document}

\title[Stable matchings]{On constrained matchings, stable under random preferences} 
\date{}
\author{Boris Pittel}
\address{Department of Mathematics, Ohio State University, 231 West 18-th Avenue, Columbus OH 43210-1175}
\email{pittel.1@osu.edu}
\begin{abstract} Colloquially, there are two groups, $n$ men and $n$ women, each man (woman) ranking women  (men) as potential marriage partners. A complete matching is called stable if no unmatched pair prefer each other to their partners in the matching. If some pairs are not admissible, then such a matching may not exist, but a properly defined  partial stable matching exists always, and all such matchings involve the same, equi-numerous, groups of men and women. Earlier we proved that, for the complete, random, preference lists, with high probability (whp) the total number of complete stable matchings is, roughly, of order $n^{1/2}$, at least. Here we consider the case that the preference lists are still complete, but a generic pair (man,woman) is admissible with probability $p$, independently of all other $n^2-1$ pairs. It is shown that the
expected number of complete stable matchings tends to $0$ if, roughly, $p<\tfrac{\log^2 n}{n}$ and to infinity if  $p>\tfrac{\log^2 n}{n}$. We show that whp: (a) there exists a complete stable matching if $p>(9/4)\tfrac{\log^2 n}{n}$, (b) the number of unmatched men and women is bounded if $p> \tfrac{\log^2n}{n}$,
and (c) this number grows as a fractional power of $n$ for $p<\tfrac{\log^2 n}{n}$.
\end{abstract}

\subjclass{60C05; 05C05, 92B10}

\maketitle
\section{Introduction and results} The classical model of a stable matching problem is sometimes formulated in colloquial terms of marriages. An instance of such a problem involves $n$ men and $n$ women, with every member ranking all, or some of
the members of opposite sex as potential marriage partners. A complete matching of two sets, if any exists, is called stable 
if no unmatched pair (man, woman) prefer each other to their partners under matching. In a classic paper \cite{Gal}, 
Gale and Shapley proved that when all ranking lists are complete, of length $n$ that is, at least one stable matching always exists. They proved this remarkable result by analyzing an equally remarkable algorithm:  men propose to women in simultaneous rounds, and each woman temporarily accepts the best proposal, by selecting her favorite man among the current proponents and her current suitor if she has one already, while each of the rejected men proposes, in the next round,  to a favorite woman among women whom he hasn't proposed to yet. The algorithm terminates once each woman holds a proposal;  the resulting matching is stable. And of course the roles can be reversed: women propose and men choose. It is well known (Gusfield and Irving \cite{Gus}) that the ``men propose/women choose'' algorithm delivers a men-optimal stable matching: at every other stable matching each wife is not higher on her husband's preference list than in the terminal matching for this algorithm. And likewise,  the ``women propose/men choose'' algorithm delivers a women-optimal stable matching.

A decade later, McVitie and Wilson \cite{Mac} suggested a sequential algorithm: men propose and women choose like in Gale-Shapley process, but the proposals are made in sequence, one proposal at a step. This fundamemtal algorithm determines the same, men-optimal, stable matching, as the Gale-Shapley algorithm does. Soon after Wilson \cite{Wil} proved that the average running time, i.e. the
total number of proposals by men is bounded by $nH_n$, $H_n=\sum_{j\in [n]}\tfrac{1}{j}$, if it is assumed  that all $2n$  preference lists are uniformly random, independent permutations of the opposite sex members. The ingenious proof was based on observation that the total number of proposals is {\it stochastically dominated\/} by the duration of a classic {\it coupon collecting\/} process: at each step a ball is put uniformly at random at one of $n$ distinguishable boxes, until each box is non-empty. Knuth \cite{Knu} undertook a systematic study of the stable matchings, showing in particular that the expected running time of the  McVitie-Wilson
algorithm is actually asymptotic to $nH_n\sim n\log n$. Consequently, the average rank of the wife  from the men-optimal matching is asymptotic to $\log n$. 

The book \cite{Knu} turned out to be highly influential, thanks to a list of thought-provoking, open problems. One of them was to estimate the expected value of $S_n$, the total number of stable matchings for the random instance of the matching problem. Using an inclusion-exclusion formula followed-up by an ingenious interpretation of each alternating term as a value of a multi-dimensional integral, Knuth found that 
\begin{equation}\label{1}
\Bbb E[S_n]=n! P_n,\quad P_n=\int\limits_{\bold x,\bold y\in [0,1]^n}\prod_{1\le i\neq j\le n}(1-x_iy_j)\,d\bold x\,d\bold y.
\end{equation}
The roots of this remarkable formula are clear: $n!$ is the total number of bijections between the set of men and the set of women, so $P_n$ has to be, and it is, the probability that a generic bijection is a stable matching under the independently uniform preference lists. 

We used the equation \eqref{1}  in \cite{Pit1} to prove that $E[S_n]\sim e^{-1}n\log n$, as $n\to\infty$.  Twenty years later
Lennon and the author \cite{Len} proved an analogous, more complex, formula for $\Bbb E[S_n^2]$, and it was used to show that
$\Bbb E[S_n^2]\sim(e^{-2}+(2e^3)^{-2})(n\log n)^2$, implying that $S_n$ exceeds $\Bbb E[S_n]$ with probability 
$>\bigl(1+(2e)^{-1}\bigr)^{-1}$.  We also proved \cite{Pit2} that whp $S_n$ is of order $\sqrt{n\log^{-1} n}$, at least.

There is an alternative proof of \eqref{1}, which we used for many other other extensions, see \cite{Pit1} and  \cite{Pit2}, 
for example.
 Let us assume that there are given two $n\times n$ matrices $X=\{X_{i,j}\}$ and $Y=\{Y_{i,j}\}$ with the rows and the columns corresponding to the men set $U=[n]$ and the women set $V=[n]$, respectively, whose entries are all independent, each uniformly distributed on
$[0,1]$. For each man $i$ (woman $j$) we define a permutation, i.e. ordering $\pi_i$ (respectively $\omega_j$) of $V$ ($U$ respectively) obtained by listing the women (men) in the increasing order of the entries in the $i$-th row of $X$
(in the $j$-th column of $Y$, respectively). Thus, the $2n$ preference lists are these permutations; by the definition of $X$ and $Y$, they are uniformly random, and mutually independent. 

Consider the diagonal matching formed by pairing man $i\in U$ and woman $i\in V$. Conditioned on $X_{k,k}=x_k$ and $Y_{k,k}=y_k$, $(k\in [n])$, the probability that this matching is stable is given by
\[
P_n(\bold x,\bold y)=\prod_{1\le i\neq j\le n}(1-x_iy_j).
\]
Indeed, for $i\neq j$,  $1-x_iy_j$ is the conditional probability that the unmatched pair $(i,j)$ of man $i$ and woman $j$ do not 
prefer each other to their partners $i\in V$ and $j\in U$ under the diagonal matching, and these $n^2$ events are all conditionally independent.  And \eqref{1} follows by Fubini's theorem. Here is a useful extension of \eqref{1}, \cite{Pit1}.
Let $P_{n,k}$ denote the probability that the diagonal matching is stable {\it and\/} the total rank of wives $i\in V$ (husbands $i\in U$) as ranked by their respective husbands $i\in U$ (wives $i\in V$) is $k$; so $k\in [n,n^2]$. Then 
\begin{equation}\label{2}
P_{n,k}=\int\limits_{\bold x,\bold y\in [0,1]^n}[z^{k-n}]\prod_{1\le i\neq j\le n}\bigl(1-x_i(1-z+zy_j)\bigr)\,d\bold x d\bold y;
\end{equation}
the integrand is the coefficient by $z^{k-n}$ in the product. This formula implies \eqref{1} since the sum of the integrands over $k$ is the value of the last product at $z=1$.

When some of the preference lists are incomplete, so that some pairs $(\text{man},\text{woman})$ are unacceptable, a complete stable matching may not exist. Gusfield and Irving \cite{Gus} 
define a possibly incomplete matching as unstable ``if there is a man $m$ and a woman $w$ such that  {\bf (i)\/} $m$ and $w$ are not partners in $M$, but each is acceptable to the other; {\bf (ii)\/} $m$ is either unmatched in $M$, or prefers $w$ to his partner in $M$; {\bf (iii)\/} $w$ is either unmatched in $M$, or prefers $m$ to her partner in $M$''. They proved 
that in this case ``the men and the women are each partitioned into two sets---those that have partners in all stable matchings and those that have partners in none''.

Recently Stoyan Dimitrov and Adam Hammett \cite{Dim} posed a problem of stable matchings with incomplete random preference lists. Supposing that that the preference list of each man (woman) consists of $k$ women ($k$ men), and otherwise the lists are 
uniformly random and independent, what is the probability that a complete stable matching exists? 
What is the threshold value of $k(n)$ for likely existence of a complete stable matching? Stoyan and Adam's question 
jolted the author's memory: it was proved in \cite{Pit2} that in the case of unrestricted preference lists the maximum rank 
of a wife,  when men propose and  women select, is likely to be of order $\log^2 n$. 
However, contrary to a casual remark in \cite{Pit2}, it does not necessarily mean that the critical $k(n)$ is of order $\log^2 n$.

Unlike the above model, in this paper we continue to stay with the case of complete preference lists, though this time the prospective pairs must each pass an ``admissibility'' test. (We leave it to the reader to come up with their own colloquial interpretation of such a test. Romeo and Juliet story, as a poetic inspiration?) For a chance of sharp analysis, we need to preserve homogeneity of the model and mutual independence of admissibility tests for individual pairs. To this end, we make the simplest assumption, namely that the results of the tests for various pairs form a sequence of independent trials with success (admissibility) probability equal to a given $p\in (0,1)$. What is the threshold value $p=p(n)$ for likely existence of a complete stable matching? 

The main result in this paper is that $p(n)$ is exactly of order $\tfrac{\log^2 n}{n}$. For the critical $p(n)$, the likely length of a shortest preference list is of order $\tfrac{\log^2 n}{\log\log n}$, at least, which makes it plausible that the critical $k(n)$ for Stoyan and Adam's question is also of order $\tfrac{\log^2 n}{\log\log n}$, at least.

To put this research into perspective, a classic
result of Erd\H os and R\'enyi \cite{Erd} is that the critical value $p(n)$ for likely existence of a perfect matching in the random bipartite graph on $[n]\times [n]$ with independent edge probabilities $p$ is $\tfrac{\log n}{n}$.

\subsection{Integral formulas for the $p$-extension of the matching problem.} Let us see why the posed problem holds a promise. Begin with $\Bbb E[S_n]$, where  $S_n=S_n(p)$ stands for the total number of {\it complete\/} stable matchings. We use again the two random matrices $X$ and $Y$ to generate the uniformly random, independent preference lists. The diagonal matching is stable if {\bf (a)\/} all $n$ pairs $(i,i)$ pass the admissibility test;  {\bf (b)\/}
for each $i\neq j$ it is not true that man $i$ and woman $j$ prefer each other to their partners, woman $i$ and man $j$ respectively, {\it and\/} the unmatched pair $(i,j)$ passes the admissibility test. Conditioned on $X_{k,k}=x_k$ and $Y_{k,k}=y_k$, $(k\in [n])$, the $(n)_2$ events described above are independent, so the conditional probability of stability
is 
\begin{equation*}
P_n(x,y)=P_{n,p}(x,y):=p^n\prod_{1\le i\neq j\le n}(1-px_iy_j),
\end{equation*}
whence
\begin{equation}\label{3}
\Bbb E[S_n]=n! P_n,\quad P_n=p^n\!\!\!\!\!\!\!\!\int\limits_{\bold x,\bold y\in [0,1]^n}\prod_{1\le i\neq j\le n}(1-px_iy_j)\,d\bold x\,d\bold y.
\end{equation}
(Does the reader see that $P_n$ increases with $p$, as intuitively it should?) Turn to the $p$-counterpart of the equation \eqref{2}. 
First of all, the total rank of wives in the diagonal matching $M$, call it $Q_n$, is given by 
\[
Q_n=n+\sum_{i\in [n]}\big|\{j:\, (i,j) \text{ is admissible}, X_{i,j}<X_{i,i}\bigr\}\big|.
\]
The $i$-th term is the total number of admissible women for man $i$, whom he prefers to his partner, woman $i$.
The task is to evaluate 
$P_{n,k}(\bold x,\bold y)$ the conditional probability of the event $\{M\text{ is stable}\}\cap\{Q_n=k\}$. Using $\Bbb I(A)$ to denote the event $A$'s indicator, and $\circ$ to denote conditioning on $\{X_{i,i}=x_i, Y_{i,i}=y_i\}_{i\in [n]}$, we have
\begin{multline*}
P_{n,k}(\bold x,\bold y)\\
=\Bbb E\biggl[\Bbb I(M\text{is stable})\cdot\Bbb I\Bigl(\sum_{i\in [n]}\big|\bigl\{j: (i,j)\text{ admissible}, X_{i,j}<x_i\bigr\}\big|=k-n\Bigr)\Big|\circ\biggr]\\
=[\xi^{k-n}]\,\Bbb E\bigl[\Bbb I(M\text{is stable})\xi^{Q_n-n}\big|\circ\bigr].\\
\end{multline*}
To evaluate this expression, we turn to the probabilistic view of a generating function. Fix $\xi\in (0,1)$; sift through the 
admissible pairs  $(i\neq j)$ and, whenever $X_{i,j}<x_i$, ``mark'' the pair $(i,j)$ with probability $\xi$, independently of all other unmatched, admissible pairs. Then the last expectation is the conditional probability that $M$ is stable and all the admissible pairs $(i,j)$ with $X_{i,j}<x_i$ are marked. Consequently
\[
\Bbb E\bigl[\Bbb I(M\text{is stable})\xi^{Q_n}\big|\circ\bigr]=p^n\, \Bbb P(B|\circ),\quad B:=\bigcap_{1\le i\neq j\le n}B_{i,j};
\]
here 
\begin{align*}
B_{i,j}&=(X_{i,j}>x_i)\bigcup\Bigl(\bigl(X_{i,j}<x_i; (i,j)\text{ non-admissible or }(i,j)\text{ admissible}\\
&\qquad\qquad\qquad\qquad\text{ and }Y_{i,j}>y_j\bigr)\cap \bigl((i,j)\text{ is marked}\bigr)\Bigr).
\end{align*}
The events $B_{i,j}$ are conditionally independent, and
\begin{multline*}
\Bbb P(B_{i,j}|\circ)= 1- x_i+ x_i\bigl[1-p +p(1-y_j)\xi\bigr]=1-px_i(1-\xi+\xi y_j).\\
\end{multline*}
Therefore
\[
P_{n,k}(\bold x,\bold y)=p^n [\xi^{k-n}]\prod_{1\le i\neq j\le n}\bigl(1-px_i(1-\xi+\xi y_j)\bigr),
\]
whence 
\begin{equation}\label{4}
\Bbb P\bigl(M\text{ stable}, Q_n=k\bigr)=p^n\!\!\!\!\!\!\!\int\limits_{\bold x, \bold y\in [0,1]^n}
\!\!\!\![\xi^{k-n}]\prod_{1\le i\neq j\le n}\!\!\!\!\bigl(1-px_i(1-z+zy_j)\bigr)\,d\bold x d\bold y.
\end{equation}
Of course, a similar equation holds for the total rank of husbands in $M$ among all admissible men.
 \subsection{The $p$-extension of the proposal algoritnm.} Let $\mathcal M$ and $\mathcal W$ denote the set of men and the set of women, $|\mathcal M|=|\mathcal W|=n$. Consider ``the men propose/the women choose" algorithm: at each step a currently non-aligned man $m$ proposes to the woman $w$, his {\it best\/} choice among $n$ women to whom he hasn't proposed yet,
and $m$ and $w$ become a couple with a given probability $p>0$,  
either if $w$ does not have a partner, or she does, but prefers $m$ to him. A rejected proponent will have to resume proposing. Once a woman holds a proposal, she may be  proposed to by other men, each temporarily accepted (admissible) man being her {\it best\/} choice among the previously accepted proponents. That is, the set of women holding proposals can only increase. As for men, each of them has kept moving {\it down\/} on his preference list. The process stops after at most $n^2$ proposals, when there is {\it no admissible\/} pair $(m,w)$ such that:

\noindent {\bf (i)\/} $m$ and $w$ are both matched separately, but prefer each other to their partners; (if that had been the case, $m$ would have earlier proposed to, and formed a couple with $w$, and $w$ would have ended with a different partner, higher on her preference list);

\noindent or {\bf (ii)\/} $m$ is unmatched and $w$ is either unmatched, or matched, but prefers $m$ to her partner; (if that 
had been the case, $w$ would have been matched at the end with a man who is at least as high on her list as $m$, thus different from $w$'s partner);

\noindent or {\bf (iii)\/}: $m$ is matched and $w$ is unmatched, but $m$ prefers $w$ to his partner; (if that had been the case, $m$ would have proposed to, and been accepted by $w$ earlier in the process, who would have ended with a partner).

So, the proposal algorithm delivers two  equinumerous sets, a set $T_1$ of men and a set $T_2$   
of women, each $w\in T_2$ holding a proposal from a man $m\in T_1$, while the remaining men and women are unmatched. 
The properties {\bf (i)\/}-{\bf (iii)\/} qualify $T_1, T_2$ and the matching between them as a ({\it possibly partial\/}) stable matching, and $T_1$ and $T_2$ are exactly the sets of men and the women that are matched in every stable matching, Gusfield and Iriving \cite{Gus}.

Let $T_1\subseteq \mathcal M$,  $T_2\subseteq \mathcal W$, $|T_1|=|T_2|=:\ell$. Let $M$ denote a matching between $T_1$ and $T_2$. Given $\bold x:=\{x_m\}_{m\in T_1}\in [0,1]^{T_1}$, $\bold y:=\{y_w\}_{w\in T_2}\in [0,1]^{T_2}$, we need to bound 
$\Bbb P(M|\bold x,\bold y)$, the probability that $M$ is the terminal matching in the proposal algorithm, {\it conditioned\/} 
on the event$\{X_{m,M(m)}=x_m,\,Y_{M(w),w}=y_w\}_{m\in T_1,w\in T_2}$. To upper-bound $\Bbb P(M|\bold x,\bold y)$, consider the constraints imposed on  all pairs $(m,w)$ for the event in question to hold.

{\bf (a)\/} If $m\in T_1,\,w\in T_2$ and $(m,w)\in M$, i.e. $w=M(m)$ ($m=M(w)$), then $(m,w)$ is admissible, an event of probability $p$.

{\bf (b)\/} If $m\in T_1$, $w\in T_2$ and $w\neq M(m)$, then either
$(m,w)$ is inadmissible (event of probability $1-p$), or $(m,w)$ is admissible (event of probability $p$). In the latter case,
either $m$ prefers $M(m)$ to $w$, or $m$ prefers $w$ to $M(m)$. 
The (conditional) probability of this event is
\[
1-p+p\bigl[\Bbb P(X_{m,w}> x_m)+\Bbb P(X_{m,w}<x_m)\cdot\Bbb P(Y_{m,w}>y_w)\bigr]=1-px_my_w.
\]
(The reader may prefer a sleeker argument: the complementary event for a pair $(m,w)$ in question is ``$(m,w)$ is admissible, and $m$ and $w$ prefer each other to their partners $M(m)$ and $M(w)$''.)

\noindent{\bf (c')\/} If $m\in T_1$, $w\notin T_2$, then either $(m,w)$ is inadmissible, or $(m,w)$ is admissible, but $m$ prefers $M(m)$ to $w$, which is the event of probability $1-p +p(1-x_m)=1-px_m$.

\noindent {\bf (c'')\/} If $m\notin T_1$, $w\in T_2$, then either $(m,w)$ is inadmissible, or $(m,w)$ is admissible, but $w$ prefers $M(w)$ to $m$, which is the event of probability $1-p+p(1-y_w)=1-py_w$,

\noindent {\bf (d)\/} If $m\notin T_1$, $w\notin T_2$, then $(m,w)$ is inadmissible, the event of probability $1-p$.

\noindent Since the events in items {\bf (a)-(d)\/} are (conditionally) independent, we obtain
\begin{multline}\label{6}
\Bbb P(M|\bold x,\bold y)\le 
p^{\ell}\prod_{m\in T_1,\, w\in T_2\atop w\neq M(m)}(1-px_my_w)\\
\times \biggl(\prod_{m\in T_1}(1-px_m)\cdot\prod_{ w\in T_2}(1-py_w)\biggr)^{n-\ell}\cdot (1-p)^{(n-\ell)^2}.
\end{multline}
Introduce $P_n(M)$, the probability that the proposal algorithm delivers a partial matching $M$. By \eqref{6}, we obtain
\begin{multline*}
P_n(M)=\int\limits_{\bold x,\bold y\in [0.1]^{\ell}}\Bbb P(M|\bold x,\bold y)\,d\bold x d\bold y
\le p^{\ell}(1-p)^{(n-\ell)^2}\\
\times\int\limits_{\bold x,\bold y\in [0,1]^{\ell}}\prod_{m\in T_1, w\in T_2\atop w\neq M(m)}(1-px_my_w)
\times \biggl(\prod_{m\in T_1}(1-px_m)\cdot\prod_{ w\in T_2}(1-py_w)\biggr)^{n-\ell}\,d\bold xd\bold y.
\end{multline*}
We emphasize that this is inequality. Let $P_n(\ell)$ be the probability that the proposal algorithm delivers a matching of $\ell$ men and $\ell$ women. By  the union bound, and homogeneity, the bound for $P_n(M)$ implies 
\begin{multline}\label{6.5}
P_n(\ell)\le p^{\ell}(1-p)^{(n-\ell)^2}\binom{n}{\ell}\binom{n}{\ell}\ell!\\
\times\int\limits_{\bold x,\bold y\in [0.1]^{\ell}}\prod_{i\in [\ell], j\in [\ell]\atop i\neq j}(1-px_iy_j)
\cdot \biggl(\prod_{i\in [\ell]}(1-px_i)\cdot\prod_{j\in [\ell]}(1-py_j)\biggr)^{n-\ell}\,d\bold xd\bold y.
\end{multline}
This bound does not reflect the intrinsic asymmetry of the proposal algorithm, and 
 the RHS of \eqref{6.5} is the expected number of the stable partial matchings involving $\ell$ men and $\ell$ women. 
Keeping only the first product in the above integrand, and using \eqref{3}, we obtain then
\begin{equation}\label{7}
P_n(\ell)\le (1-p)^{(n-\ell)^2}\binom{n}{\ell}^2 \Bbb E[S_{\ell}].
\end{equation}
Here $\Bbb E[S_{\ell}]$ is the expected number of the complete stable matchings between the men set $[\ell]$ and the women set $[\ell]$.

To proceed,  for $m\in T_1$ and $w\in \mathcal W\setminus T_2$ either $(m,w)$ is not admissible or $m$ prefers his partner in $T_2$. So the rank of $m$'s partner among all admissible women in $\mathcal W$ is the rank of $m$'s partner among admissible women in $T_2$. Likewise, for $w\in T_2$ the rank of $w$'s partner among all admissible men in $\mathcal M$ is the rank of $w$'s partner among admissible men in $T_1$. This observation leads to the following hybrid
of  \eqref{4} and \eqref{6.5}. Given $\ell$, let $M$ be the diagonal matching of two copies of $[\ell]$, and $Q$
Then, for $k\ge \ell$,
\begin{multline}\label{6.6}
\Bbb P(M\text{ stable}, Q=k)\\
=p^{\ell}(1-p)^{(n-\ell)^2}\!\!\!\!\!\!\int\limits_{\bold x, \bold y\in [0,1]^{\ell}}\!\!\!\bigl[\xi^{k-\ell}\bigr]
\prod_{1\le i\neq j\le \ell}\bigl(1-px_i(1-z+zy_j)\bigr)\\
\times \biggl(\prod_{i\in [\ell]}(1-px_i)\cdot\prod_{j\in [\ell]}(1-py_j)\biggr)^{n-\ell}\,d\bold x d\bold y.
\end{multline}

\subsection{Critical value of $p$.}
\begin{theorem}\label{thm1} {\bf (i)\/} 
If  $p\le \tfrac{\log^2 n}{n}\bigl(1-\a\tfrac{\log\log n}{\log n}\bigr)$, $\a>2$, then $\Bbb E[S_n]\to 0$. {\bf (ii)\/} If 
$p\ge  \tfrac{\log^2 n}{n}$,  then $\Bbb E[S_n]\to \infty$.
\end{theorem} 

The proof is based on the properties of a random partition of $[0,1]$. Let $X_1,\dots, X_{\ell}$ denote the independent random variables, each distributed uniformly on $[0,1]$. Set $\mathcal S_{\ell}=\sum_{j\in [\ell]}X_j$, $V_j=\tfrac{X_j}{\mathcal S_{\ell}},\, (j\in [\ell])$, so that $V_j\in [0,1]$, $\sum_{j\in [\ell]}V_j=1$. Introduce $f_{\ell}(s,v_1,\dots,v_{\ell-1})$, the density of 
$(\mathcal S_{\ell},V_1,\dots V_{\ell-1})$, which is non-zero for $v_j\le \tfrac{1}{s}$, $\sum_{j<\ell}v_j\le 1$. The Jacobian of $\bold x=(x_1,\dots, x_{\ell})$ with respect to $(s,v_1,\dots, v_{\ell-1})$ is $s^{\ell-1}$, implying that
\begin{equation}\label{A}
 f_{\ell}(s,v_1,\dots,v_{\ell-1})=s^{\ell-1}\,\Bbb I\Bigl(\sum_{j<\ell}v_j\le 1\text{ and } \max_{j\le \ell}v_j\le\tfrac{1}{s}\Bigr),
\end{equation}
where $v_{\ell}:=1-\sum_{j<\ell} v_j$. Now, $(\ell-1)!\cdot\Bbb I\Bigl(\sum_{j<\ell}v_j\le 1\Bigr)$ is the joint density of $L_1,\dots,L_{\ell-1}$,  the lengths of the first $(\ell-1)$ consecutive intervals obtained by selecting  $(\ell-1)$ points 
independently and uniformly at random in $[0,1]$. So, integrating \eqref{A} over $v_1,\dots, v_{\ell-1}$, we obtain: 
the density of $S_{\ell}$ is given by
\begin{equation}\label{B}
f_{\ell}(s)=\tfrac{s^{\ell-1}}{(\ell-1)!}\cdot \Bbb I \bigl(\max_{j\in [\ell]}L_j\le \tfrac{1}{s}\bigr), 
\end{equation}
cf. Feller \cite{Fel}. Dropping the constraint  $\max_{j\le \ell}v_j\le\tfrac{1}{s}$ in \eqref{A}, we obtain a crucial 
inequality
\begin{equation}\label{C}
 f_{\ell}(s,v_1,\dots,v_{\ell-1})\le \tfrac{s^{\ell-1}}{(\ell-1)!}\cdot g_{\ell}(v_1,\dots, v_{\ell-1}),
 \end{equation}
 where $g_{\ell}(v_1,\dots, v_{\ell-1})$ is the density of $L_1,\dots, L_{\ell-1}$.
 
To use this connection to $L_1,\dots, L_{\ell}$, we will need asymptotic properties of $\L_{\ell}^+=\max_{j\in [\ell]}L_j$ and $U_{\ell}:=\sum_{j\in [\ell]}L_j^2$:
\begin{lemma}\label{lem2} As $\ell\to\infty$, in probability $\tfrac{L_{\ell}^+}{\ell^{-1}\log \ell}\to 1$, and $\ell\, U_{\ell}\to 2$. More sharply, for every $\rho>0$, we have 
\[
\Bbb P\bigl(L_{\ell}^+\le \ell^{-1}(\log\tfrac{\ell}{\log \ell}-\rho)\bigr)=O(\ell^{-d}), \quad
\forall\,d\in (0,e^{\rho}-1).,
\]
and for $\omega(\ell)\to\infty$ however slowly 
\[
\Bbb P\bigl(L_{\ell}^+\le \ell^{-1}\log(\ell\omega(\ell))\bigr)=(1+o(1))\exp\bigl(-\tfrac{1+o(1)}{\omega(\ell)}\bigr)\to 1, 
\]
while, for every $\delta<\tfrac{1}{3}$,
\[
\Bbb P\biggl(\big|\tfrac{\ell\, U_{\ell}}{2}-1\big|\ge \ell^{-\delta}\biggr)=O\bigl(e^{-\Theta(\ell^{1/3-\delta})}\bigr).
\]
\end{lemma}
\noindent {\bf Note.\/} 
The proof of  Lemma \ref{lem2}, is based on the following  fact: $\{L_j: j\in [\ell]\}$ has the same distribution as $\bigl\{W_j/\sum_{i\in [\ell]} W_i: j\in [\ell]\bigr\}$, where $W_1,\dots,W_{\ell}$ are i.i.d. exponentials, with the common parameter $1$, say. The large deviation estimate for $\ell\, U_{\ell}$ was proved in Pittel \cite{Pit3}.
\subsubsection{Proof of Theorem \ref{thm1}, part {\bf (i)\/}.}  Since $\Bbb E[S_n]$ is increasing with $p$, it suffices to consider $p=p(n)=\tfrac{\log^2 n}{n}\bigl(1-\a\tfrac{\log\log n}{\log n}\bigr)$, $\a>2$. By \eqref{3}, $\Bbb E[S_n]=n! P_n,$ where
\[
P_n=p^n\!\!\!\!\!\!\!\!\int\limits_{\bold x,\bold y\in [0,1]^n}\prod_{1\le i\neq j\le n}(1-px_iy_j)\,d\bold x\,d\bold y.
\]
So, using $1-\xi\le \exp(-\xi-\xi^2/2)$, $\xi\in (0,1)$, and denoting $s_j=\sum_{i\neq j}x_i$, $t_j=\sum_{i\neq j}x_i^2$, we have
\begin{equation}\label{9}
P_n\le p^n\int\limits_{\bold x\in [0,1]^n}\biggl(\prod_{j\in [n]}\int_0^1\exp\bigl(-pys_j-\tfrac{p^2y^2t_j}{2}\bigr)\,dy\biggr)
d\bold x.
\end{equation}
Consider $\int_1$, the contribution to the RHS of \eqref{9} from $\{\bold x\in [0,1]^n: s\le s(n)\}$, $s:=\sum_{i\in [n]}x_i$, $s(n):=\tfrac{\log n-\be\log\log n}{p}$, $\be\in (1,\a)$. Dropping $\exp\bigl(-\tfrac{p^2y^2t_j}{2}\bigr)$, and integrating over $y\in [0,1]$, we have 
\[
\int_1\le p^n\int\limits_{\bold x: s\le s(n)}\prod_{j\in [n]}F(s_jp)\,d\bold x,\quad F(z):=\tfrac{1-e^{-z}}{z}.
\]
Now,
\begin{equation*}
(\log F(z))'=\tfrac{1}{e^z-1}-\tfrac{1}{z}\left\{\begin{aligned}
&\to -\tfrac{1}{2},&&z\downarrow 0,\\
&\sim -\tfrac{1}{z},&&z\uparrow\infty,
\end{aligned}\right.
\end{equation*}
so that, for $A:=\sup_{z>0}(z+1)\bigl(\tfrac{1}{z}-\tfrac{1}{e^z-1}\bigr)$, we have
\begin{equation}\label{9.1}
\Big|(\log F(z))'\Bigr|\le \tfrac{A}{z+1},
\end{equation}
Consequently
\begin{multline*}
\log F(s_jp)=\log F(sp)-\int\limits_{s-x_j}^s(\log F(yp))'_y\,dy\\
\le \log F(sp)+\tfrac{Apx_j}{(s-x_j)p+1}\le \log F(sp)+\tfrac{Ax_j}{s},
\end{multline*}
whence
$
\prod_{j\in [n]}F(s_jp)\le e^A F^n(sp). 
$

Then, by \eqref{B}, 
\begin{multline*}
\int_1\le e^{A}p^n\int_0^{s(n)}\biggl(\tfrac{1-e^{-sp}}{sp}\biggr)^n\tfrac{s^{n-1}}{(n-1)!}\,ds=\tfrac{e^A}{(n-1)!}\int_0^{ps(n)}\exp(\Phi_n(\eta))\,d \eta, \\
\end{multline*}
$\Phi_n(\eta):=n\log(1-e^{-\eta})-\log\eta$. Here, as $\tfrac{e^{\eta}-1}{\eta}$ increases, 
\begin{align*}
\Phi'_n(\eta)&=\tfrac{n}{e^{\eta}-1}-\tfrac{1}{\eta}=\tfrac{n}{e^{\eta}-1}\bigl(1-\tfrac{e^{\eta}-1}{n\eta}\bigr)\\
&\ge \tfrac{n}{e^{\eta}-1}\bigl(1-\tfrac{n\log^{-\be}n-1}{n(\log n-\be\log\log n)}\bigr)=\tfrac{n}{e^{\eta}-1}(1-O(\log^{-\be-1}n))>0,
\end{align*}
if $n$ sufficiently large, and, as $\tfrac{e^{\eta/2}-e^{-\eta/2}}{\eta}$ increases,
\begin{align*}
\Phi_n''(\eta)&=-\tfrac{n}{(e^{\eta/2}-e^{-\eta/2})^2}+\tfrac{1}{\eta^2}\\
&\le-\tfrac{n}{(e^{\eta/2}-e^{-\eta/2})^2}\bigl(1-\tfrac{\bigl(e^{(\log n-\be\log\log n)/2}-e^{-(\log n-\be\log\log n)/2)}\bigr)^2}{n(\log n-\be\log\log n)^2}\bigr)\\
&=-\tfrac{n}{(e^{\eta/2}-e^{-\eta/2})^2}(1-O(\log^{-2-\be}n))<0,
\end{align*}
if $n$ is sufficiently large. So, that $\Phi'_n(\eta)>0$ and $\Phi_n''(\eta)<0$ for $\eta\in [0,1]$, and $n$ sufficiently large, i.e. $\Phi_n(\eta)$ is concave and decreases on $[0,ps(n)]$. Therefore, for $\eta\in [0,ps(n)]$, we have
$
\Phi_n(\eta)\le \Phi_n(1)+\Phi_n'(1)(\eta-1).
$
Consequently,
\begin{align*}
\int_0^{ps(n)}\exp(\Phi_n(\eta))\,d \eta&\le\tfrac{e^{A}}{(n-1)!} \exp[\Phi_n(ps(n))]
\int_0^1\exp\bigl(\Phi_n'(ps(n))(\eta-1)\bigr)\,d\eta\\
&\le \tfrac{e^{A}}{(n-1)!}\tfrac{\exp[\Phi_n(ps(n))]}{\Phi_n'(ps(n))}.
\end{align*}
So, $n! \int_1$, the contribution of $s\le s(n)$ to $\Bbb E[S_n] (=n! P_n)$, is of order
\begin{multline*}
n\tfrac{(1-e^{-ps(n)})^n}{p(s(n))}\cdot\tfrac{e^{ps(n)}-1}{n}\le \exp\bigl[-n e^{-ps(n)}+ps(n)\bigr]\\
\le \exp\bigl(-ne^{-(\log n-\be\log\log n)}+\log n\bigr)= \exp\bigl(-\log^{\be}n+\log n\bigr)\to 0,
\end{multline*}
since $\be>1$.

Turn to the contribution to the RHS of \eqref{9} from $\bold x$ with $s=\sum_{i\in [n]} x_i\ge s(n)$. The $j$-th integral on the RHS of \eqref{9} is bounded by 
\begin{align*}
\tfrac{1}{s_j}\int_0^{\infty}\exp\bigl(-pz-\tfrac{p^2t_jz^2}{2s^2}\bigr)\,dz
&=\tfrac{1}{s_j}\biggl[\tfrac{1}{p}-\tfrac{pt_j}{s^2}\int_0^{\infty}z\exp\bigl(-pz-\tfrac{p^2t_j z^2}{2s^2}\bigr)\,dz\biggr]\\
&\le\tfrac{1}{s_j}\biggl[\tfrac{1}{p}-\tfrac{pt_j}{s^2}\int_0^{\infty}z\exp\bigl(-pz-\tfrac{p^2 t z^2}{2s^2}\bigr)\,dz\biggr].
\end{align*}
So, 
\begin{multline*}
 p^n\biggl(\prod_{j\in [n]}\int_0^1\exp\bigl(-pys_j-\tfrac{p^2y^2t_j}{2}\bigr)\,dy\biggr)\\
 \le \prod_{j=1}^n\tfrac{1}{s_j}\cdot\prod_{j\in [n]}\biggl[1-\tfrac{p^2t_j}{s^2}\int_0^{\infty}z\exp\bigl(-pz-\tfrac{p^2t z^2}{2s^2}\bigr)\,dz\biggr]\\
 \le s^{-n} \exp\biggl(-\tfrac{p^2}{s^2}\Bigl(\sum_{j\in [n]}t_j\Bigr)\int_0^{\infty}z\exp\bigl(-pz-\tfrac{p^2t z^2}{2s^2}\bigr)\,dz\biggr)
 \prod_{j\in [n]}\bigl(1-\tfrac{x_j}{s}\bigr)^{-1}.
 \end{multline*}
Here
\begin{align*}
\sum_{j\in n]}t_j &= (n-1)t=(n-1)\sum_{j\in [n]}x_j^2,\\
\prod_{j\in [n]}\bigl(1-\tfrac{x_j}{s}\bigr)^{-1}&=\exp\biggl(\sum_{j\in [n]}\bigl(\tfrac{x_j}{s}+O(\tfrac{x_j^2}{s^2})\bigr)\biggr)
=\exp\bigl(1+O(p)\bigr).
\end{align*}
Therefore
\begin{multline*}
 p^n\biggl(\prod_{j\in [n]}\int_0^1\exp\bigl(-pys_j-\tfrac{p^2y^2t_j}{2}\bigr)\,dy\biggr)\\
\le s^{-n}\exp\bigl(1+O(p)\bigr)\exp\biggl(\!-\tfrac{p^2}{s^2}(n-1)t\int_0^{\infty}z\exp\bigl(-pz-\tfrac{p^2t z^2}{2s^2}\bigr)\,dz\!\biggr)\\
=O\biggl[s^{-n}\exp\biggl(\!-\tfrac{(n-1)t}{s^2}\int_0^{\infty}\eta\exp\bigl(-\eta-\tfrac{t \eta^2}{2s^2}\bigr)\,d\eta\!\biggr)\biggr].
\end{multline*}
By convexity of the exponential function, and $\int_{\eta\ge 0}\eta e^{-\eta}\,d\eta=1$, the last integral is at least
\[
\exp\biggl(-\int_0^{\infty} e^{-\eta}\tfrac{t\eta^3}{2s^2}\,d\eta\biggr)=\exp\bigl(-\tfrac{6t}{2s^2}\bigr).\\
\]
Therefore, uniformly for $s=\sum_{j\in [n]}x_j\ge s(n)$, 
\[
p^n\biggl(\prod_{j\in [n]}\int_0^1\exp\bigl(-pys_j-\tfrac{p^2y^2t_j}{2}\bigr)\,dy\biggr)
=O\Bigl(s^{-n}\exp\bigl(-\tfrac{(n-1)t}{s^2}\exp\bigl(-\tfrac{6t}{2s^2}\bigr)\bigr)\Bigr).
\]
Importantly, the last bound is the product of two functions, first dependent only on $s=\sum_{i\in [n]} x_i$, and second dependent only on $t/s^2=\sum_{i\in [n]}\tfrac{x_i^2}{s^2}$. So, by \eqref{C},  the contribution 
of $\{\bold x\in[0,1]^n: s\in [s(n),n]\}$ to the RHS of \eqref{9} is at most of order
\begin{multline*}
\tfrac{1}{(n-1)!}\int\limits_{s=s(n)}^n\!\!\!\!\tfrac{1}{s}\,\Bbb E\biggl[\Bbb I\bigl(L_n^{+}\le \tfrac{1}{s}\bigr)
\exp\Bigl(-(n-1)U_n\cdot e^{-3U_n}\Bigr)\biggr]\,ds\\
\le \tfrac{\log n}{(n-1)!}\Bbb P\bigl(L_n^+\le \tfrac{1}{s(n)}\bigr)=\tfrac{\log n}{(n-1)!}\,\Bbb P\bigl(L_n^+\le \tfrac{p}{\log n-\be\log\log n}\bigr).
\end{multline*}
where $L_n^{+}=\max_{i\in [n]} L_i$ and $U_n=\sum_{i\in [n]}L_i^2$. Here
\begin{multline*}
\tfrac{p}{\log n-\be\log\log n}=\tfrac{\log n}{n}\cdot \tfrac{\log n-\a\log\log n}{\log n-\be\log\log n}\\
=n^{-1}\Bigl(\log n+(\be-\a)\log\log n+O\bigl(\tfrac{(\log\log n)^2}{\log n}\bigr)\Bigr)\\
=n^{-1}\Bigl(\log\tfrac{n}{\log n}+(\be-\a+1+o(1))\log\log n\Bigr)
\end{multline*}
Since $\a>2$, we can pick $\be\in (1,\a-1)$, making $\be-\a+1<0$. Therefore, by Lemma \ref{lem2},
\[
\Bbb P\bigl(L_n^+\le \tfrac{p}{\log n-\be\log\log n}\bigr) = O(n^{-d}),\quad \forall\, d>0.
\]
So, the contribution to the RHS of \eqref{9} from $\bold x$ with $s=\sum_{i\in [n]} x_i\ge s(n)$ is of order
$n!\cdot\tfrac{\log n}{(n-1)!} n^{-d}$, i. e. super-polynomially small. 

We conclude that $\Bbb E[S_n]$ is super-polynomially small as well. The proof of Theorem \ref{thm1} {\bf (i)\/}
is complete.

\subsubsection{Proof of Theorem \ref{thm1}, part {\bf (ii)\/}.} This time we need to sharply bound $P_n$, given by \eqref{3}, from below, which we do by discarding appropriately chosen {\it peripheral\/} parts of the cube $[0,1]^n= \{\bold x: x_j\in [0,1]\}$. And these are the parts that were, in essence, proven to be negligible in the preceding argument.
 
 Let $D$ be defined by the constraints: $\bold x\ge \bold 0$, and 
\begin{equation}\label{11}
\begin{aligned}
&\tfrac{n}{\log(n\omega_1(n))}\le s\le \tfrac{n}{\log(n\omega_2(n))},\\
&s^{-1}x_j\le\tfrac{\log(n\omega_2(n))}{n},\quad j\in [n],\\
&\tfrac{t}{s^2}\le \tfrac{3}{n},
\end{aligned}
\end{equation}
where $\omega_1(n)>\omega_2(n)\to\infty$ are to be specified later. So $D\subset [0,1]^n$, since each $x_j\le 1$. Set
\begin{equation}\label{12}
\hat P_n=p^n\int_{\bold x\in D}\biggl(\prod_{j\in [n]}\biggl(\int_0^1\prod_{i\neq j}(1-px_iy_j)\,dy_j\biggr)\biggr)\,d\bold x;
\end{equation}
so $P_n\ge \hat P_n$. For $\bold x\in D$, $\bold y\in [0,1]^n$, we have $px_iy_j=O(n^{-1}\log^2 n)$. Therefore, denoting
$s_j=\sum_{i\neq j}x_i$, $t_j=\sum_{i\neq j}x_j^2$,
\begin{multline*}
\prod_{i\neq j}(1-px_iy_j)= 
\exp\bigl(-py_js_j-\tfrac{p^2y_j^2t_j}{2}\bigr)\cdot \exp\bigl(O(n^{-2}\log^4 n)\bigr).\\
\end{multline*}
So, 
\begin{multline*}
\prod_{j\in [n]}\biggl(\int_0^1\prod_{i\neq j}(1-px_iy_j)\,dy_j\biggr)\\
\ge\exp\bigl(O(n^{-1}\log^4 n)\bigr)
\biggl(\int_0^1\exp\bigl(-pys-\tfrac{p^2y^2t}{2}\bigr)\,dy\biggr)^n\\
\ge(1+o(1)) (ps)^{-n}\biggl(\int_0^{ps}\exp\bigl(-\eta-\tfrac{\eta^2 t}{2s^2}\bigr)\,d\eta\biggr)^n\\
\ge (1+o(1)) (ps)^{-n}\biggl(\int_0^{ps}\exp\bigl(-\eta-\tfrac{2\eta^2}{n}\bigr)\,d\eta\biggr)^n,
\end{multline*}
since $t/s^2\le 3$ for $\bold x\in D$. By convexity of the exponential function, the last integral exceeds
\begin{multline*}
(1-e^{-ps})\exp\biggl(-\tfrac{2/n}{1-e^{-ps}}\int_0^{ps}e^{-\eta}\eta^2\,d\eta\biggr)\\
=(1-e^{-ps})\exp\biggl(\tfrac{-4/n\bigl[e^{ps}-(1+(ps)+(ps)^2/2\bigr]}{e^{ps}-1}\biggr)
\end{multline*}
Since $s\ge \tfrac{n}{\log(n\omega_1(n))}$, we have $ps\to\infty$ if, say, $\tfrac{\log \omega_1(n)}{\log n}\to 0$. For this $\omega_1(n)$, we obtain then
\[
\int_0^{ps}\exp\bigl(-\eta-\tfrac{2\eta^2}{n}\bigr)\,d\eta\ge (1-e^{-ps}) e^{-4/n}.
\]
Consequently
\[
\prod_{j\in [n]}\biggl(\int_0^1\prod_{i\neq j}(1-px_iy_j)\,dy_j\biggr)\ge e^{-4}\bigl(\tfrac{1-e^{-ps}}
{ps}\bigr)^n,
\]
implying that 
\begin{equation}\label{12}
\hat P_n\ge e^{-4}\int_{\bold x\in D}\bigl(\tfrac{1-e^{-ps}}{s}\bigr)^n\,d\bold x.
\end{equation}
i. e. the integrand is a function of $s$ only. It remains to bound this integral from below. Let us switch to the new variables: $u=\sum_{j\in [n]}x_j(=s)$, $v_j=x_j/s$, $\,j\in [n-1]$. Define also $v_n=x_n/s$, so that $v_j\in [0,1]$ and $\sum_{j\in [n]}v_j=1$. The conditions \eqref{11} become
\begin{equation}\label{13}
\begin{aligned}
&\tfrac{n}{\log(n\omega_1(n))}\le u\le \tfrac{n}{\log(n\omega_2(n))},\\
 &v_j\le \tfrac{\log (n\omega_2(n))}{n},\,j\in [n],\\
 & \sum_{j\in [n]} v_j^2\le \tfrac{3}{n}.
\end{aligned}
\end{equation}
So, the region is the Cartesian product $\bigl[\tfrac{n}{\log(n\omega_1(n))},\tfrac{n}{\log(n\omega_2(n))}\bigr]\times D^*$, where the interval is the range of $u$, and $D^*$ is the range of $\bold v$. The Jacobian of $\bold x$ with respect to $(u,v_1,\dots,v_{n-1})$ is $u^{n-1}$. Recall that $(n-1)!$ is the joint density of $L_1,\dots,L_{n-1}$,
the first $(n-1)$ intervals obtained by selecting independently and uniformly $(n-1)$ points in $[0,1]$. Then, the equation
\eqref{13} is equivalent to
\begin{multline}\label{14}
\hat P_n\ge \tfrac{1}{(n-1)!}\biggl(\int_{\tfrac{n}{\log(n\omega_1(n))}}^{\tfrac{n}{\log(n\omega_2(n))}}
\tfrac{(1-e^{-pu})^n}{u}\,du\biggr)\Bbb P(\bold L\in D^*),\\
\end{multline}
where
\begin{multline*}
=\Bbb P(\bold L\in D^*)=\Bbb P\bigl(\max_{j\in [n]}\tfrac{\log(n\omega_2(n))}{n};\, \sum_{j\in [n]}L_j^2\le \tfrac{3}{n}\bigr) 
\to 1,
\end{multline*}
see Lemma \ref{lem2}. 
Pick $\omega_1(n)=o(\log n)$, and $\omega_2(n)=\tfrac{\omega_1(n)}{2}$. It is easy to see that 
\[
\int_{\tfrac{n}{\log(n\omega_1(n))}}^{\tfrac{n}{\log(n\omega_2(n))}}
\tfrac{1}{u}\,du
\]
is asymptotic to $\tfrac{\log 2}{\log n}$. Furthermore, since $p=\tfrac{\log^2 n}{n}$, 
\[
pu\ge \tfrac{pn}{\log(n\omega_1(n))}=\frac{\log n}{1+\tfrac{\log \omega_1(n)}{\log n}},
\]
and we have
\begin{multline*}
(1-e^{-pu})^n=\exp\bigl(-(1+o(1))ne^{-pu}\bigr)\\
=\exp\biggl(\!-(1+o(1))\exp\Big(\tfrac{\log\omega_1(n)}{1+\tfrac{\log\omega_1(n)}{\log n}}\Bigr)\biggr)\ge \tfrac{1}{\omega_1^2}\ge \log^{-2}n.\\
\end{multline*} 
Hence $\hat P_n\gtrsim\tfrac{e^{-4}\log 2}{(n-1)!\log^3n}$, and using $\Bbb E[S_n]\ge n!\hat P_n$, we conclude that $\Bbb E[S_n]\gtrsim \tfrac{e^{-4}n\log 2}{\log^3 n}\to\infty$. The proof of Theorem \ref{thm1} {\bf (ii)\/} is complete.
\subsection{Likely size of a stable partial matching.}
\begin{theorem}\label{thm2}  Let $\mathcal L_n$ denote the size of the stable partial matching delivered by the $p$-proposal algorithm,
that is the cardinality of the set of men and of the set  of women matched in every stable partial matching. 
If $p=p(n)\gg 1/n$, then for $\ga>2$ 
\[
\Bbb P\bigl(\mathcal L_n\ge n -\tfrac{\ga\log(np)}{p}\bigr)\ge 1-n^{-(1+o(1))(\ga^2-2\ga)\log(np)}.\\
\]
Thus for $\ga>2$, with probability super-polynomially close to $1$, all but at most $\tfrac{\ga\log(np)}{p}$ men and $\tfrac{c\log(np)}{p}$ women, i. e. of order $o(n)$ each,  are matched.
\end{theorem}
\begin{proof} By \eqref{7},
\begin{equation}\label{15}
\Bbb P\bigl(\mathcal L_n\le n -\tfrac{\ga\log(np)}{p}\bigr)\le\!\!\! \sum_{\ell\le n -\tfrac{\ga\log(np)}{p}}\!\!\!(1-p)^{(n-\ell)^2}\binom{n}{\ell}^2\Bbb E[S_{\ell}].
\end{equation}
Here $\Bbb E[S_{\ell}]$ is the expected number of stable matchings for $\ell$ men and $\ell$ women, {\it and\/} $p=1$,
which is of order $O(\ell\log \ell)=O(n\log n)$. Furthermore, 
\[
(1-p)^{j^2}\binom{n}{j}^2\le e^{f(j)},\quad f(x):=-x^2p+2x\log\tfrac{en}{x}.
\]
So,
\[
f\bigl(\tfrac{\ga\log(np)}{p}\bigr)=-\tfrac{\ga^2\log^2(np)}{p}+\tfrac{2\ga\log (np)}{p}\log\tfrac{en}{\tfrac{\ga\log(np)}{p}}
\le -\tfrac{\ga^2-2\ga}{p}\log^2(np).
\]
Next, $f(x)$ is concave, and 
\[
f'\bigl(\tfrac{\ga\log(np)}{p}\bigr)=-\tfrac{2p\ga\log(np)}{p}+2\log\tfrac{np}{\ga\log(np)}
\le -2(\ga-1)\log(np)\to-\infty.\\
\]
Therefore, the RHS of \eqref{15} is of order
\[
(n\log n) \exp\bigl( -\tfrac{\ga^2-2\ga}{p}\log^2(np)\bigr).
\]
Here, since $\tfrac{\log x}{x}$ decreases for $x\ge e$,
\[
\tfrac{\log^2(np)}{p}=n\log(np)\cdot\tfrac{\log(np)}{np}\ge \log n\cdot\log(np).
\]
Therefore 
\begin{multline*}
(n\log n) \exp\bigl( -\tfrac{\ga^2-2\ga}{p}\log^2(np)\bigr)\le (\log n)\exp\bigl(\log n-(\ga^2-2\ga)\log n\cdot \log(np)\bigr)\\
=(\log n)\cdot n^{-(\ga^2-2\ga)\log(np)+1}.
\end{multline*}
\end{proof}

\begin{theorem}\label{thm3} If $p\ge c\tfrac{\log^2 n}{n}$, $c>1+\tfrac{1}{2\Delta}$, then
$\Bbb P(\mathcal L_n\le n-\Delta)\to 0$. In words, if $p>c\tfrac{\log^2 n}{n}$, $c>1$, then the number of unmatched men and women is bounded in probability. In particular, for $c>\tfrac{9}{4}$, then whp $\mathcal L_n=n$.
\end{theorem}

\begin{proof}
 By Theorem \ref{thm2} and \eqref{7}, 
\begin{equation}\label{15.01}
\Bbb P(\mathcal L_n\le n-\Delta)\le n^{-\Theta(\log(np))}+\sum_{\ell=n-cp^{-1}\log(np)}^{n-\Delta}p^{\ell}\binom{n}{\ell}^2\ell! \,I _{n,\ell}.
\end{equation}
Here, denoting $s_j=\sum_{i\in [\ell]\setminus \{j\}}x_j$, $s=\sum_{i\in [\ell]}x_i$, and using $1-z\le e^{-z}$,  we bound
\begin{multline}\label{15.02}
I_{n,\ell}=\!\!\!\!\int\limits_{\bold x,\bold y\in [0,1]^{\ell}}\!\prod_{1\le i\neq j\le \ell}\!\!(1-px_iy_j)\prod_{i\in [\ell]}\!(1-px_i)^{n-\ell}\cdot \prod_{j\in [\ell]}(1-py_j)^{n-\ell}\,d\bold x \,d\bold y\\
\le\int\limits_{\bold x,\,\bold y\in [0,1]^{\ell}}\exp\biggl(-p(n-\ell)s -\sum_{j\in [\ell]}y_j(ps_j+p(n-\ell))\biggr)\,d\bold x\,d\bold y\\
=\int\limits_{\bold x\in [0,1]^{\ell}}e^{-p(n-\ell)s}\,\prod_{j\in [\ell]}\tfrac{1-e^{-ps_j-p(n-\ell)}}{ps_j+p(n-\ell)}\,\,d\bold x\\
\le e^{A}\int\limits_{s\in [0,\ell]}e^{-p(n-\ell)s}\biggl(\tfrac{1-e^{-ps-p(n-\ell)}}{ps+p(n-\ell)}\biggr)^{\ell}\cdot\tfrac{s^{\ell-1}}{(\ell-1)!}\,ds\\
\le \tfrac{e^{A}}{p^{\ell}(\ell-1)!}\int\limits_0^{p\ell}e^{-(n-\ell)u}\biggl(\tfrac{u(1-e^{-u-p(n-\ell)})}{u+p(n-\ell)}\biggr)^{\ell-1}\,du
\\
\le \tfrac{e^{A}}{p^{\ell}(\ell-1)!}\int\limits_0^{p\ell}\exp\Bigl[-(n-\ell)\Bigl(u+\tfrac{p(\ell-1)}{u+p(n-\ell)}\Bigr)\Bigr]\,du.
\end{multline}
The factor $e^A$ came from \eqref{9.1}, and we used $\tfrac{u}{u+p(n-\ell)}\le \exp\bigl(-\tfrac{p(n-\ell)}{u+p(n-\ell)}\bigr)$. We bound the bottom integral using Laplace method. Observe that, denoting by $\mathcal H(u)$ the logarithm of the integrand, 
\begin{align*}
\mathcal H'_u&=-(n-\ell)\Bigl(1-\tfrac{p(\ell-1)}{(u+p(n-\ell))^2}\Bigr),\\
\mathcal H''_u&=-(n-\ell)\tfrac{2p(\ell-1)}{(u+p(n-\ell))^3},
\end{align*}
implying concavity of $\mathcal H$. It follows that $\mathcal H$ attains its maximum at 
\[
\bar u=\bar u(\ell)=\sqrt{p(\ell-1)}-p(n-\ell),
\]
and that 
\begin{align*}
\mathcal H(\bar u)&=-(n-\ell)\Bigl(2\sqrt{p(\ell-1)}-p(n-\ell)\Bigr)=-2(n-\ell)\sqrt{np}(1+o(1)), \\
\mathcal H''(\bar u)&=-\tfrac{2(n-\ell)}{\sqrt{p(\ell-1)}},
\end{align*}
since $p(n-\ell) =O(\log (np))=o(\sqrt{pn})$.  So, approximating the bottom integrand in \eqref{15.02} by $\exp\bigl(\mathcal H(\bar u)+\tfrac{\mathcal H''(\bar u)}{2}(u-\bar u)^2\bigr)$, we easily obtain that the integral itself  is of order $\frac{\exp(\mathcal H(\bar u))}{\sqrt{-\mathcal H''(\bar u)}}$. Consequently, the $\ell$-th term on the RHS of \eqref{15.01} is of order
\begin{multline*}
p^{\ell}\ell!\cdot\tfrac{n^{2(n-\ell)}}{p^{\ell} (\ell-1)!}\tfrac{\exp\bigl(-2(n-\ell)\sqrt{np}(1+o(1))\bigr)}{\sqrt{\tfrac{n-\ell}{\sqrt{p(\ell-1)}}}}\\
=O\Bigl[n (pn)^{1/4}\cdot\exp\bigl[2(n-\ell)\bigl(\log n - \sqrt{np}(1+o(1))\bigr]\Bigr]\\
=O\Bigl[n\exp\bigl[2(n-\ell)\bigl(\log n-(1+o(1))c^{1/2}\log n\bigr)\bigr]\Bigr],\\
\end{multline*}
if $p\ge c\tfrac{\log^2 n}{n}$. So, for $c>1$, the sum over $\ell\le n-\Delta$ on the RHS of \eqref{15.01} is of
order
\begin{multline*}
n \cdot\exp\bigl[2\Delta\bigl(\log n-(1+o(1))c^{1/2}\log n \bigr)\bigr]\\=O\Bigl(\cdot\exp\bigl[(2\Delta+1 -(1+o(1)))2 \Delta c^{1/2}\log n\bigr]\Bigr)\to 0,
\end{multline*}
if $c> \bigl(1+\tfrac{1}{2\Delta}\bigr)^{1/2}$. The proof of Theorem \ref{thm3} is complete.
\end{proof}

\begin{theorem}\label{thm3.5} For $p\le c\tfrac{\log^2 n}{n}$, $c<1$, and $\Delta_n=n^{1-\sqrt{c}} \log^{-2} n$, we have $\Bbb P(\mathcal L_n\ge n-\Delta_n)\to 0$ super-polynomially fast,
so that whp at least $\Delta_n$ men and women remain unmatched.
\end{theorem}
\begin{proof} By \eqref{7},
\begin{equation}\label{15.2}
\Bbb P\bigl(\mathcal L_n\ge n -\Delta_n\bigr)\le\!\!\! \sum_{\ell\ge  n -\Delta_n}\binom{n}{\ell}^2\Bbb E[S_{\ell}]\le
n^{2\Delta_n}\sum_{\ell\ge  n -\Delta_n}\Bbb E[S_{\ell}].
\end{equation}
The condition $p\le c\tfrac{\log^2 n}{n}$ is equivalent to $p\le c(\ell)\tfrac{\log^2\ell}{\ell}$ for $c(\ell):=c\tfrac{n^{-1}\log^2 n}{\ell^{-1}\log^2\ell}\le c$, for $\ell\ge n-\Delta_n$. Notice at once that $c(\ell)=c\exp(O((n-\ell)/n))=c(1+O(n^{-\sqrt{c}}))$. Now, 
\[
\Bbb E[S_{\ell}]=\ell! P_{\ell},\quad P_{\ell} =p^n\!\!\!\!\!\!\!\!\int\limits_{\bold x,\bold y\in [0,1]^{\ell}}\prod_{1\le i\neq j\le \ell}(1-px_iy_j)\,d\bold x\,d\bold y.
\]
Consider $\int_1$, the contribution to $P_{\ell}$ from $\{\bold x\in [0,1]^{\ell}: s\le s(\ell)\}$, $s(\ell):=\tfrac{\sigma \log\ell}{p}$, where $\sigma\in (c(\ell),1)$. Arguing closely to the corresponding part in the proof of Theorem \ref{thm1}, part {\bf (i)\/}, we
obtain that $\ell! \int_1$ is of order 
\begin{multline*}
\ell\tfrac{(1-e^{-ps(\ell)})^n}{p(s(\ell))}\cdot\tfrac{e^{ps(\ell)}-1}{n}\le \exp\bigl[-\ell e^{-ps(\ell)}+ps(\ell)\bigr]\\
\le \exp\bigl(-\ell^{1-\sigma}+\log \ell\bigr)\le \exp\bigl(-0.5\ell^{1-\sigma}\bigr).
\end{multline*}
Following the second part in the proof of Theorem \ref{thm1}, part {\bf (i)\/}, we obtain that the contribution to $P_{\ell}$ 
from $\{\bold x\in [0,1]^{\ell}: s\ge s(\ell)\}$ is at most of order
\begin{multline*}
\tfrac{1}{(\ell-1)!}\int\limits_{s=s(\ell)}^{\ell}\!\!\!\!\tfrac{1}{s}\,\Bbb E\biggl[\Bbb I\bigl(L_{\ell}^{+}\le \tfrac{1}{s}\bigr)
\exp\Bigl(-(\ell-1)U_{\ell}\cdot e^{-3U_{\ell}}\Bigr)\biggr]\,ds\\
\le \tfrac{\log \ell}{(\ell-1)!}\Bbb P\bigl(L_{\ell}^+\le \tfrac{1}{s(\ell)}\bigr)=\tfrac{\log \ell}{(\ell-1)!}\,\Bbb P\bigl(L_{\ell}^+\le \tfrac{p}{\sigma\log\ell}\bigr)\le \tfrac{\log \ell}{(\ell-1)!}\,\Bbb P\bigl(L_{\ell}^+\le\tfrac{c(\ell)}{\sigma}\cdot \tfrac{\log\ell}{\ell}\bigr),
\end{multline*}
where $L_{\ell}^{+}=\max_{i\in [\ell]} L_i$ and $U_{\ell}=\sum_{i\in [\ell]}L_i^2$. 
So, the contribution to $\Bbb E[S_{\ell}]$  from $\bold x\in [0,1]^{\ell}$ with $s\ge s(\ell)$ is of order
$\ell (\log \ell) \exp\bigl(-\Theta(\ell^{1-c(\ell)/\sigma})\bigr)=\exp\bigl(-\Theta(\ell^{1-c(\ell)/\sigma})\bigr)$
Therefore $\Bbb E[S_{\ell}]$ is of order 
$
\exp\bigl(-\Theta(\ell^{\min(1-\sigma,1-c(\ell)/\sigma)})\bigr)$, and we get the best estimate 
$\exp\bigl[-\Theta\bigl(\ell^{1-\sqrt{c(\ell)}}\bigr)\bigr]$ by selecting $\sigma=\sqrt{c(\ell)}$. 
So, using \eqref{15.2}, $\sqrt{c(\ell)}=(1+O(n^{-\sqrt{c}}))\sqrt{c}$, and $\Delta_n=n^{1-\sqrt{c}} \log^{-2} n$, we conclude that 
\[
\Bbb P\bigl(\mathcal L_n\ge n -\Delta_n\bigr)\le\exp\Bigl(2\Delta_n\log n-\Theta\bigl(n^{1-\sqrt{c}}\,\bigr)\Bigr)
=\exp\Bigl(-\Theta\bigl(n^{1-\sqrt{c}}\,\bigr)\Bigr),
\]
The proof of Theorem \ref{thm3.5} is complete.
\end{proof}

\subsection{Asymptotic ranks of stable partners.} Let $\mathcal Q_n^{(-)}$, $\mathcal Q_n^{(+)}$ denote the smallest total rank and the largest total rank of wives in a (possibly partial) stable matching among all admissible women partners. They are equi-distributed with, respectively, 
 $\mathcal R_n^{(-)}$, $\mathcal R_n^{(+)}$, the extreme total ranks of husbands among admissible men in a stable matching. 
 \begin{theorem}\label{thm4} Suppose that $p=c\tfrac{\log^2 n}{n}$, $c<1$. For $\eps\in (0,1)$ we have:
whp $\lfloor (1-\eps)(pn^3)^{1/2}\rfloor\le \mathcal Q_n^{(-)}\le \mathcal Q_n^{(+)}\le  \lfloor (1+\eps)(pn^3)^{1/2}\rfloor$. In words, for $p$ strictly below $\tfrac{\log^2 n}{n}$, whp the total rank of  stable wives (husbands, resp.) among admissible partners is asymptotically the same for every stable matching.
\end{theorem}
{\bf Note.\/} In a sharp contrast, for unconstrained matchings we proved  \cite{Pit2} that whp the pairs $(Q(M), R(M))$
of ranks of wives and husbands for stable matchings $M$ are all densely concentrated around a {\it hyperbola\/} $QR=n^3$, connecting $\bigl(n\log n, \tfrac{n^2}{\log n}\bigr)$ and $\bigl(\tfrac{n^2}{\log n}, n\log n\bigr)$, the asymptotic rank values for men-optimal and women-optimal matchings. Ashlagi, Kanoria and Leshno \cite{Ash} had discovered 
a similar asymptotic independence $(Q(M), R(M))$ of a stable matching $M$   
when the numbers of men and women differ even by $1$ only. In a follow-up paper \cite{Pit4}, we proved asymptotic concentration of the scaled 
$(Q(M), R(M))$ around a certain pair $(Q,R)$ dependent on the numbers of men and women.
\begin{proof} {\bf (i)\/} Let $k=\big \lfloor(1-\eps)(pn^3)^{1/2}\rfloor$.  By \eqref{6.5}, \eqref{6.6} and the union bound, we have:
for $\ell_n:=\lfloor n-\ga p^{-1}\log(np)\rfloor,\,\ga >2$, and $\Delta_n=n^{1-\sqrt{c}} \log^{-2} n$,
\begin{multline}\label{16}
\Bbb P\bigl(\mathcal Q_n^{(-)}\le k\bigr)\le \Bbb P\bigl(\mathcal L_n\notin [\ell_n, n-\Delta_n)\bigr)+\sum_{\ell\in [\ell_n, n-\Delta_n)}p^{\ell}(1-p)^{(n-\ell)^2}\binom{n}{\ell}^2 \ell !\\
\times \int\limits_{\bold x,\bold y\in [0,1]^{\ell}}\sum_{m=\ell}^k\bigl[\xi^{m-\ell}\bigr]\prod_{1\le i\neq j\le \ell}\bigl[1-px_i(1-\xi+\xi y_j)\bigr]\,d\bold x d\bold y\\
\times\biggl(\prod_{i\in [\ell]}(1-px_i)\cdot\prod_{j\in [\ell]}(1-py_j)\biggr)^{n-\ell}\,d\bold x d\bold y.
\end{multline}
By Theorem \ref{thm2} and Theorem \ref{thm3.5}, the probability term on the RHS is super-polynomially small. Furthermore, the bottom product is at most 
\begin{equation}\label{16.01}
\exp(-p(n-\ell)s) \exp\biggl(-p(n-\ell)\sum_{j\in [\ell]}y_j\biggr),\quad s:=\sum_{i\le\ell} x_i.
\end{equation}
Since the remaining double product is a polynomial of $\xi$ with non-negative, $(\bold x,\bold y)$-dependent coefficients, we apply
the Chernoff-type bound and obtain
\begin{multline}\label{!-}
\sum_{m=\ell}^k\bigl[\xi^{m-\ell}\bigr]\!\!\prod_{1\le i\neq j\le \ell}\!\!\bigl[1-px_i(1-\xi+\xi y_j)\bigr]
\!\le\! \inf_{\xi\le 1}\xi^{\ell-k}\!\!\!\prod_{1\le i\neq j\le \ell}\!\!\!\bigl[1-px_i(1-\xi+\xi y_j)\bigr]\\
\le\inf_{\xi\le 1}\exp\biggl((\ell-k)\log\xi-\sum_{j\le\ell} ps_j(1-\xi+\xi y_j)\biggr),\quad \Bigl(s_j:=\sum_{i\neq j}x_i\Bigr),\\
=\inf_{\xi\le 1}\exp\biggl((\ell-k)\log\xi-(\ell-1)(1-\xi)ps-\xi\sum_{j\le\ell} ps_j y_j\biggr).
\end{multline}
Using \eqref{16.01} and integrating over $\bold y\in [0,1]^{\ell}$, we bound the resulting integral in \eqref{16} by
\begin{equation}\label{!}
\int\limits_{\bold x\in [0,1]^{\ell}}\!\!\inf_{\xi\le 1}\exp\Bigl[(\ell-k)\log\xi-\bigl(n-1-\xi(\ell-1)\bigr)ps\Bigr]
\prod_{j\le\ell}\tfrac{1-e^{-p(\xi s_j+n-\ell)}}{
p(\xi s_j+n-\ell)}\,d\bold x
\end{equation}
Since $\xi\le 1$, by \eqref{9.1} the last product is bounded by
\[
\bigl(\tfrac{1-e^{-p(\xi s+n-\ell)}}{p(\xi s+n-\ell)}\bigr)^{\ell}\cdot\exp\Bigl(A\sum_{j\le\ell}\tfrac{p\xi x_j}{p(\xi s_j+n-\ell)+1}\Bigr)\\
\le e^A\cdot \bigl(\tfrac{1-e^{-p(\xi s+n-\ell)}}{p(\xi s+n-\ell)}\bigr)^{\ell}.
\]
Thus, replacing $d\bold x$ with $\tfrac{s^{\ell-1}\, ds}{(\ell-1)!}$, we see that the integral in \eqref{16} is at most of order
\begin{multline*}
\tfrac{1}{(\ell-1)!}\int\limits_{s\le \ell}\inf_{\xi\le 1}\exp\Bigl[(\ell-k)\log\xi-\bigl(n-1-\xi(\ell-1)\bigr)ps\Bigr]
\Bigl(\tfrac{1-e^{-p(\xi s+n-\ell)}}{
p(\xi s+n-\ell)}\Bigr)^{\ell}\,s^{\ell-1}ds\\
=\tfrac{1}{p^{\ell}(\ell-1)!}\int\limits_0^{p\ell}\,\inf_{\xi\le 1}\exp\bigl[(\ell-k)\log\xi-\bigl(n-1-\xi(\ell-1)\bigr)u\bigr]\,
\biggl(\tfrac{1-e^{-\xi u-p(n-\ell)}}{\xi u+p(n-\ell)}\biggr)^{\ell} u^{\ell-1}\,du.\\
\end{multline*}
To get the most out of this upper bound, we need to determine $\xi=\xi(u)\le 1$ that nearly minimizes 
\[
H(\xi,u):=(\ell-k)\log\xi+\xi(\ell-1)u +\ell \log\tfrac{1-e^{-\xi u-p(n-\ell)}}{\xi u+p(n-\ell)},  
\]
for $\xi\in  (0,1]$. As a function of $\xi\in (0,\infty)$,  $H(\xi,u)$ is convex and attains its minimum at a point relatively close to $\xi=\tfrac{k/\ell}{u}$.
So, we {\it define\/} $\xi(u)=\tfrac{k/\ell}{u}$ for
$u\ge k/\ell$, and $\xi(u)\equiv1$ for $u<k/\ell$.
For $\xi=\xi(u)$, we bound the above integral by $\int_1+\int_2$, where
\[
\int_1=\int\limits_0^{k/\ell} e^{-(n-\ell)u}\,\biggl(\tfrac{1-e^{-u-p(n-\ell)}}{u+p(n-\ell)}\biggr)^{\ell-1}u^{\ell-1}\,du,
\]
and
\begin{multline}\label{16.02}
\int_2=\biggl(\tfrac{1-e^{-k/\ell-p(n-\ell)}}{k/\ell+p(n-\ell)}\biggr)^{\ell}\bigl(\tfrac{k}{\ell}\bigr)^{\ell-k}
\exp\bigl(\tfrac{k(\ell-1)}{\ell}\bigr)
\int\limits_{k/\ell}^{p\ell} u^{k-1} e^{-u(n-1)}\,du\\
\le \bigl(\tfrac{\ell}{k}\bigr)^{\ell}\exp\Bigl(-\tfrac{p(n-\ell)\ell^2}{k}\bigl(1+O(k^{-1}n\log(np))\bigr)\Bigr)
\bigl(\tfrac{k}{\ell}\bigr)^{\ell-k}
\exp\bigl(\tfrac{k(\ell-1)}{\ell}\bigr)\cdot\tfrac{(k-1)!}{(n-1)^k}\\
\le \text{const}\,k^{-1/2}\exp\Bigl(-(n-\ell)\tfrac{pn^2}{k}\bigl(1+O((np)^{-1/2}\log(np))\bigr)\Bigr)
\bigl(\tfrac{\ell}{n}\bigr)^k\\
\le\text{const }p^{-1/2}n^{-3/4}\exp\Bigl(-(n-\ell)\bigl[\tfrac{pn^2}{k}(1+o(1))+\tfrac{k}{n}\bigr]\Bigr).
\end{multline}
(For the second to last step we used Stirling formula for $k!$ and $\tfrac{k}{\ell}-\tfrac{k}{n}=O\bigl(\tfrac{\log (np)}{\sqrt{np}}\bigr)$.) So, the contribution of $\bold x$ with $u=ps\ge k/\ell$ to $\ell$-th term in \eqref{16} is of order 
\begin{multline*}
p^{\ell}\cdot\tfrac{1}{p^{\ell}(\ell-1)!}\binom{n}{\ell}^2 \ell ! \,p^{-1/2}n^{-3/4}\exp\Bigl(-(n-\ell)\bigl[\tfrac{pn^2}{k}(1+o(1))+\tfrac{k}{n}\bigr]\Bigr). \\
\le \tfrac{n^{1/4}}{p^{1/2}} \cdot \bigl(\tfrac{en}{n-\ell}\bigr)^{2(n-\ell)} \exp\Bigl(-(n-\ell)\bigl[\tfrac{pn^2}{k}(1+o(1))+\tfrac{k}{n}\bigr]\Bigr).\\
\end{multline*}
(We used $\binom{b}{a}\le \bigl(\tfrac{eb}{a}\bigr)^a$.) Since $n-\ell\ge \Delta_n$, and $k=\lfloor (1-\eps)(pn^3)^{1/2}\rfloor$, the last bound is of order
\begin{multline*}
\exp\Bigl(-(n-\ell)\sqrt{pn}\,\bigl(\tfrac{1}{1-\eps}+(1-\eps)+o(1)\bigr)+2\sqrt{c} (n-\ell)(\log n+O(\log\log n))\Bigr)\\
=\exp\Bigl(-(n-\ell)\bigl(\tfrac{1}{1-\eps} +(1-\eps) -2 +o(1)\bigr)\log n\Bigr)\\
=\exp\Bigl(-(n-\ell)\bigl(\tfrac{\eps^2}{1-\eps}+o(1)\bigr)\log n\Bigr).
\end{multline*}
The sum of these bottom bounds over $n-\ell>\Delta_n$ is of order $\exp(-\Theta(\Delta_n))$, i. e. which is sub-exponentially small.

Consider $\int_1$. We bound  the integrand by 
\begin{equation}\label{16.03}
\psi(u):=e^{-(n-\ell)u}\biggl(\tfrac{u}{u+p(n-\ell)}\biggr)^{\ell-1};
\end{equation}
$\psi(u)$ is log-concave so that $\psi(u)\le \psi(k/\ell)\exp\bigl((u-k/\ell)(\log\psi)'(k/\ell)\bigr)$. Here
\begin{multline}\label{16.035}
\psi(k/\ell)\le e^{-(n-\ell)k/\ell} \bigl(1-\tfrac{p(n-\ell)}{k/\ell+p(n-\ell)}\bigr)^{\ell-1}\\
=\exp\Bigl[-(n-\ell)\bigl(\tfrac{k}{\ell}+\tfrac{\ell^2 p}{k}(1+O(1/\ell+p(n-\ell)\ell/k))\bigr)\Bigr]\\
=\exp\Bigl[-(n-\ell)\bigl(\tfrac{k}{\ell}+\tfrac{\ell^2 p}{k}\bigr)\bigl(1+O(\tfrac{\log(np)}{\sqrt{np}})\bigr)\Bigr].
\end{multline}
Furthermore, 
\begin{multline}\label{16.04}
(\log \psi)'(k/\ell)=-(n-\ell)+(\ell-1)\tfrac{p(n-\ell)}{k/\ell (k/\ell+p(n-\ell))}\\
=(n-\ell)\Bigl[\tfrac{p\ell^3}{k^2}-1 +O\bigl(\tfrac{\ell^2 p}{k^2}+\tfrac{\ell^4 p^2(n-\ell)}{k^3}\bigr)\Bigr]\\
=(n-\ell)\Bigl(\tfrac{p\ell^3}{k^2}-1+O\bigl(\tfrac{\log(np)}{\sqrt{np}}\bigr)\Bigr).
\end{multline}
$\frac{p\ell^3}{k^2}-1$ is {\it positive  and bounded away from $0$}, since $\ell\sim n$ and $k^2\le (1-\eps)^2 pn^3$. It follows that $\int_1$ is of order of the bottom expression in \eqref{16.035}, uniformly for $\ell_n\le \ell<n-\Delta_n-1$. So, the contribution of $\bold x$ with $u=ps\le k/\ell$ to the $\ell$-term in \eqref{16} is of order 
\[
n\bigl(\tfrac{en}{n-\ell}\bigr)^{n-\ell}\exp\Bigl[-(n-\ell)\bigl(\tfrac{k}{\ell}+\tfrac{\ell^2 p}{k}\bigr)\bigl(1+O(\tfrac{\log(np)}{\sqrt{np}})\bigr)\Bigr].
\]
Hence, for $k=\lfloor (1-\eps)(n^3p)^{1/2}\rfloor$ and $\eps\in (0,1)$, the overall contribution of those $\bold x$' to the RHS in \eqref{16} is sub-exponentially small. 
We conclude that the sum over $\ell\in \bigl[\ell_n,n-\Delta_n-1\bigr]$ on the RHS of of \eqref{16} is sub-exponentially small, as well.

{\bf (ii)\/} The proof is close to the part {\bf (i)\/}.  
Let $k=\big \lfloor(1+\eps)(pn^3)^{1/2}\rfloor$. This time we need to upper bound $\Bbb P\bigl(Q_n^{(+)}\ge k)$.
Analogously to \eqref{16}, we have: for $\ell_n:=\lfloor n-cp^{-1}\log(np)\rfloor,\,c>2$,
\begin{multline}\label{17}
\Bbb P\bigl(\mathcal Q_n^{(+)}\ge k\bigr)\le \Bbb P\bigl(\mathcal L_n\notin [\ell_n,n-\Delta_n)\bigr)+
\sum_{\ell\in [\ell_n, n-\Delta_n)}p^{\ell}(1-p)^{(n-\ell)^2}\binom{n}{\ell}^2 \ell !\\
\times \int\limits_{\bold x,\bold y\in [0,1]^{\ell}}\sum_{m=k}^{\ell^2}\bigl[\xi^{m-\ell}\bigr]\prod_{1\le i\neq j\le \ell}\bigl[1-px_i(1-\xi+\xi y_j)\bigr]\,d\bold x d\bold y\\
\times\biggl(\prod_{i\in [\ell]}(1-px_i)\cdot\prod_{j\in [\ell]}(1-py_j)\biggr)^{n-\ell}\,d\bold x d\bold y.
\end{multline}
Like \eqref{!}, the Chernoff-type upper bound for the sum in the integrand is  
\begin{multline}\label{!!}
\sum_{m=k}^{\ell^2}\bigl[\xi^{m-\ell}\bigr]\prod_{1\le i\neq j\le \ell}\bigl[1-px_i(1-\xi+\xi y_j)\bigr]\\
\le\inf_{\xi\ge 1}\exp\biggl((\ell-k)\log\xi-(\ell-1)(1-\xi)ps-\xi\sum_{j\le\ell} ps_j y_j\biggr).
\end{multline}
Combining \eqref{17} and \eqref{!!}, and integrating with respect to $y$, we bound the resulting integral in \eqref{17} by
\begin{equation}\label{!!!}
\int\limits_{\bold x\in [0,1]^{\ell}}\!\!\inf_{\xi\ge 1}\exp\Bigl[(\ell-k)\log\xi-\bigl(n-1-\xi(\ell-1)\bigr)ps\Bigr]
\prod_{j\le\ell}\tfrac{1-e^{-p(\xi s_j+n-\ell)}}{p(\xi s_j+n-\ell)}\,d\bold x,
\end{equation}
cf. \eqref{!}. The last product is of order $\bigl(\tfrac{1-e^{-p(\xi s+n-\ell)}}{p(\xi s+n-\ell)}\bigr)^{\ell}$ if $p\xi x_j\le 1$ for
$j\le\ell$. Motivated by the part {\bf (i)\/}, with this term instead of that product, we would be content to select $\xi=
\tfrac{k/\ell}{ps}$ if $\tfrac{k}{\ell}\ge ps$, and $\xi=1$, otherwise. Now, $p\xi x_j\le 1$ for $\xi=1$; however for $\xi=\tfrac{k/\ell}{ps}
\ge 1$, we need to meet the constraint $\max_{j\le \ell}\tfrac{x_j}{s}\le\tfrac{\ell}{k}$  to satisfy $p\xi x_j\le 1$. So, we define
\[
\xi=\xi(\bold x)=\left\{\begin{aligned}
&\tfrac{k/\ell}{ps},&&\text{if }ps\le\tfrac{k}{\ell},\,\max_{j\le\ell}\tfrac{x_j}{s}\le\tfrac{\ell}{k},\\
&1,&&\text{otherwise.}\end{aligned}\right.
\]
Analogously to the bound for the integral in \eqref{16},  the resulting integral in \eqref{17} is 
$\tfrac{1}{p^{\ell}(\ell-1)!}\bigl(\int_{1}+\int_{2})$, corresponding to $\bold x$ where $\xi(\bold x)>1$ and $\xi(\bold x)=1$, respectively. Clearly
\begin{align*}
\int_1\le \int_{1}^+&=\biggl(\tfrac{1-e^{-k/\ell-p(n-\ell)}}{k/\ell+p(n-\ell)}\biggr)^{\ell}\bigl(\tfrac{k}{\ell}\bigr)^{\ell-k}
\exp\bigl(\tfrac{k(\ell-1)}{\ell}\bigr)
\int\limits_0^{k/\ell} u^{k-1} e^{-u(n-1)}\,du,\\
\int_2\le \int_2^+&=\int\limits_{k/\ell}^{p\ell} e^{-(n-\ell)u}\,\biggl(\tfrac{u}{u+p(n-\ell)}\biggr)^{\ell-1}\,du\\
&+\Bbb P\bigl(\max_j L_j\ge \tfrac{\ell}{k}\bigr)\int\limits_{0}^{k/\ell} e^{-(n-\ell)u}\,\biggl(\tfrac{u}{u+p(n-\ell)}\biggr)^{\ell-1}\,du.
\end{align*}
Comparing $\int_1^+$ to $\int_{2}$ in \eqref{16.02} we see that  $\int_1^+$ is upper-bounded by the bottom RHS
in \eqref{16.02}. This implies that, the overall contribution to the RHS sum with $\ell\le n-\Delta_n$ in \eqref{17} that comes from $\bold x$'s with $\xi(\bold x)>1$ is sub-exponentially small. As for $\int_2^+$, the common integrand is the log-concave function $\psi(u)$ defined by \eqref{16.03}, over $[k/\ell,p\ell]$. By 
\eqref{16.035} and \eqref{16.04}, we have 
\begin{align*}
\psi(k/\ell)&\le\exp\Bigl[-(n-\ell)\bigl(\tfrac{k}{\ell}+\tfrac{\ell^2 p}{k}\bigr)\bigl(1+O(\tfrac{\log(np)}{\sqrt{np}})\bigr)\Bigr],\\
(\log \psi)'(k/\ell)&=(n-\ell)\Bigl(\tfrac{p\ell^3}{k^2}-1+O\bigl(\tfrac{\log(np)}{\sqrt{np}}\bigr)\Bigr),
\end{align*}
and this time the factor $\frac{p\ell^3}{k^2}-1$ is {\it negative and bounded away from $0$}, since $\ell\sim n$ and $k^2\ge (1+\eps)^2 pn^3$. It follows that $\int_2^+$ is of order 
\begin{multline*}
\exp\Bigl[-(n-\ell)\bigl(\tfrac{k}{\ell}+\tfrac{\ell^2 p}{k}\bigr)\bigl(1+O(\tfrac{\log(np)}{\sqrt{np}})\bigr)\Bigr]\\
+\Bbb P\bigl(\max_j L_j\ge \tfrac{\ell}{k}\bigr)\exp\Bigl[(n-\ell)\tfrac{k}{\ell}\bigl(1-\tfrac{p\ell^3}{k^2}+O\bigl(\tfrac{\log(np)}{\sqrt{np}}\bigr)\bigr)\\
-(n-\ell)\bigl(\tfrac{k}{\ell}+\tfrac{p\ell^2}{k}\bigr)\bigl(1+O\bigl(\tfrac{\log(np)}{\sqrt{np}}\bigr)\bigr)\Bigr], 
\end{multline*}
uniformly for $\ell_n\le \ell<n$. The second exponent equals $-2(n-\ell)\tfrac{p\ell^2}{k}\bigl(1+O(\tfrac{\log(np)}{\sqrt{np}})\bigr)$.
Since $\tfrac{p\ell^2}{k}\le \tfrac{k}{\ell}$, we obtain that $\int_2^+$ is of order $\exp\Bigl[-2(n-\ell)\tfrac{p\ell^2}{k}\bigl(1+O(\tfrac{\log(np)}{\sqrt{np}})\bigr)\Bigr]$.

So, the contribution of $\bold x$ with $u=ps\ge k/\ell$ to the $\ell$-term in 
\eqref{17} is of order 
\[
n\bigl(\tfrac{en}{n-\ell}\bigr)^{n-\ell}\exp\Bigl(-(n-\ell)\bigl(\tfrac{k}{\ell}+\tfrac{p\ell^2}{k}\bigr)\Bigr).
\]
Hence, the overall contribution of those $\bold x$' to the RHS in \eqref{17} that comes from $\ell\le n-\Delta_n$ is sub-exponentially small. 

We conclude that so is the total sum over $\ell\in [\ell_n,n-\Delta_n-1)]$ on the RHS of \eqref{17}, implying sub-exponentiality of $\Bbb P\bigl(\mathcal Q_n^{(+)}\ge k\bigr)$.

The proof of Theorem \ref{thm4} is complete.
\end{proof}

Finally,

\begin{theorem}\label{5} Suppose that $p\ge c\tfrac{\log^2n}{n}$, $\bigl(c>\tfrac{9}{4}\bigr)$. Let $\eps\in (0,1)$. Whp 
$\mathcal Q_n^{(-)}\ge (1-\eps) n\log n$. 
\end{theorem}
\begin{proof} Using \eqref{4}, analogously to \eqref{16}, we have: for $k:=\lfloor (1-\eps)n\log n\rfloor$,
\begin{multline}\label{18} 
\Bbb P\bigl(\mathcal Q_n^{(-)}\le k\bigr)\le \Bbb P\bigl(\mathcal L_n<n)\\
+p^{n}n! 
\int\limits_{\bold x,\bold y\in [0,1]^{n}}\sum_{m=n}^k\bigl[\xi^{k-n}\bigr]\prod_{1\le i\neq j\le n}\bigl[1-px_i(1-\xi+\xi y_j)\bigr]\,d\bold x d\bold y,\\
\end{multline}
Here $\lim_{n\to\infty} \Bbb P\bigl(\mathcal L_n<n)=0$, and following closely the steps of the proof of Theorem \ref{thm4}. 
we obtain that $p^n n!\times$ the double integral is at most of order
\begin{align*}
&\qquad\qquad\qquad n\int_0^{pn}\exp\bigl[\inf_{\xi\le 1} H(\xi,u)\bigr]\, du,\\
&H(\xi,u):=-k\log\xi-(n-1)(1-\xi)u+n\log(1-e^{-\xi u})-\log u.
\end{align*}
(Similarly to those steps, it is important that we consider the case $\xi\le 1$.)
Again analogously to that proof, we are content to select $\xi=\tfrac{k/n}{u}$ for $u>\tfrac{k}{n}$, and $\xi\equiv 1$ for $u\le \tfrac{k}{n}$. Then, due to log-concavity of the resulting integrand, and $k:=\lfloor (1-\eps)n\log n\rfloor$, the contribution of $u\le k/n$ is 
\begin{multline*}
n\int_0^{k/n}\exp\bigl[H(\xi,u)\bigr]\,du=n\int_0^{k/n}\tfrac{(1-e^{-u})^n}{u}\,du\\
\le n\tfrac{(1-e^{-k/n})^n}{k/n}\cdot \bigl(\tfrac{n}{e^{k/n}-1}-\tfrac{n}{k}\bigr)^{-1}=\exp\bigl(-\Theta(n^{\eps})\bigr).
\end{multline*}
Furthermore,
\begin{multline*}
n\int_{k/n}^{pn}\exp\bigl[H(\xi,u)\bigr]\,du\le  n\bigl(\tfrac{n}{k}\bigr)^k\cdot e^k \bigl(1-e^{-k/n}\bigr)^n\int_0^{\infty}\eta^k e^{-(n-1)\eta}\,d \eta\\
\le  n\bigl(\tfrac{n}{k}\bigr)^k e^k\exp\bigl(-n e^{-k/n}\bigr)\tfrac{k!}{(n-1)^k}=\exp\bigl(-\Theta(n^{\eps})\bigr).
\end{multline*}
{\bf (ii)\/}  For $k:=\big\lceil (1+\eps)\tfrac{n^2}{\log n}\big\rceil$, analogously to the part {\bf (ii)\/} of Theorem \ref{thm4},
\begin{multline}\label{18} 
\Bbb P\bigl(\mathcal Q_n^{(+)}\ge k\bigr)\le \Bbb P\bigl(\mathcal L_n<n)\\
+p^{n}n! 
\int\limits_{\bold x,\bold y\in [0,1]^{n}} \xi(\bold x)^{n-k}\prod_{1\le i\neq j\le n}\bigl[1-px_i(1-\xi+\xi y_j)\bigr]\,d\bold x d\bold y,\\
\end{multline}
where $\xi=\xi(\bold x):=\tfrac{k/n}{ps}>1$ since $pn<\tfrac{k}{n}$.
$\max_{j\le n}\tfrac{x_j}{s}\le\tfrac{n}{k}$, and $\xi(\bold x)\equiv 1$ otherwise. Since $pn\le \tfrac{n}{\log n}\le \tfrac{k}{n}$
This time
\begin{multline*}
n\int_0^{k/n}\exp\bigl[H(\xi,u)\bigr]\,du\\
\le  n\Bbb P\Bigl(\max_{i\in [n]}L_i\le \tfrac{n}{k}\Bigr)\cdot\bigl(\tfrac{n}{k}\bigr)^k e^k \bigl(1-e^{-k/n}\bigr)^n\int_0^{\infty}\eta^k e^{-(n-1)\eta}\,d \eta\\
\le   n\Bbb P\Bigl(\max_{i\in [n]}L_i\le \tfrac{n}{k}\Bigr)\cdot  \bigl(\tfrac{n}{k}\bigr)^k e^k \tfrac{k!}{(n-1)^k}
\le n^3\Bbb P\Bigl(\max_{i\in [n]}L_i\le \tfrac{n}{k}\Bigr)\\
\le n^3 \Bbb P\Bigl(\max_{i\in [n]} L_i\le \tfrac{\log n}{(1+\eps) n}\Bigr)\le n^3 \exp\Bigl(-n^{\tfrac{\eps}{1+\eps}}\Bigr)\to 0.
\end{multline*}
\end{proof}

{\bf Notes.\/}  (1) So, for $p\ge c\tfrac{\log^2 n}{n}$, the lower bound for $\mathcal Q_n^{(-)}$ does not depend on $p$. The bound qualitatively fits the asymptotic behavior of $\mathcal Q_n^{(-)}$ for two book-end cases, $p=c\tfrac{\log^2n}{n}$, $(c<1)$, and $p=1$. For the former, $\mathcal Q_n^{(-)}\sim (pn^3)^{1/2}=c^{1/2}n\log n$ (see the preceding Theorem),
and for the latter, $\mathcal Q_n^{(-)}\sim n\log n$, see \cite{Pit1}. Is the lower bound sharp? How large is $\mathcal Q_n^{(+)}$? It was proved in \cite{Pit1} that for $p=1$, $Q_n^{(+)}\sim\tfrac{n^2}{\log n}$, which far exceeds $Q_n^{(+)}$
($\sim c^{1/2}n\log n$) for $p=c\tfrac{\log^2n}{n}$, $(c<1)$.

(2) We conjecture existence of $p^*(n)\in \bigl(\tfrac{9}{4}\tfrac{\log^2 n}{n}, 1\bigr)$, a critical threshold for appearance of a hyperbolic curve, as an asymptotic approximation of $\{(Q(M), R(M))\}_{M}$.

{\bf Acknowledgment.\/}  I am grateful to Stoyan Dimitrov for a stimulating discussion of random stable matchings under
constrained preferences.
\end{document}